\documentclass[12pt,a4paper,twoside,reqno]{amsart}
\usepackage[british,UKenglish,USenglish,english,american]{babel}

\usepackage{graphicx}
\usepackage{amscd}
\usepackage{amsmath, amssymb}
\usepackage{graphics}
\usepackage{graphicx}
\usepackage{epsfig}
\usepackage{xcolor}
\usepackage{soul}
\usepackage{tikz}
\usepackage{graphicx}
\usepackage{ mathrsfs }
\usepackage{ amssymb }
\usepackage{amscd}
\usepackage{amsthm }
\usepackage[utf8]{inputenc}
\usepackage{dirtytalk}
\usepackage{multirow}
\usepackage[all]{xy}
\usepackage{faktor}

\usepackage{url, enumitem}
\usepackage[noadjust]{cite}

\numberwithin{equation}{section}
\allowdisplaybreaks

\usepackage[pdftex,%
bookmarks=true,
breaklinks=true,
bookmarksnumbered = true,
colorlinks= true,
urlcolor= blue,
anchorcolor = yellow,
citecolor=blue,
]{hyperref}

\usepackage{fancyhdr,lipsum}

\newtheorem{theorem}{Theorem}[section]

\newtheorem{corollary}[theorem]{Corollary}

\newtheorem{definition}[theorem]{Definition}

\newtheorem{proposition}[theorem]{Proposition}
\newtheorem{remark}[theorem]{Remark}

\begin{document}

\title{Multi-welded twin groups}

\author{Mohamad N. Nasser}
\address{Mohamad N. Nasser\\
         Department of Mathematics and Computer Science\\
         Beirut Arab University\\
         P.O. Box 11-5020, Beirut, Lebanon}        
\email{m.nasser@bau.edu.lb}

\author{Oscar Ocampo} 
\address{Oscar Ocampo\\  Departamento de Matem\'atica - IME\\ Universidade Federal da Bahia\\ Av.~Milton Santos~S/N, CEP:~40170-110 - Salvador - BA - Brazil}
\email{oscaro@ufba.br}



\date{\today}

\maketitle

\begin{abstract}
For $k\geq 1$ and $n\geq 2$, we introduce the multi-welded twin group $M_kWT_n$, a natural welded analogue of the multi-virtual twin group. We show that $M_kWT_n$ arises naturally as a quotient of the universal welded braid group $UW_n(k)$, placing it within the unified framework of universal virtual and welded braid-type groups. We establish natural quotient maps relating $M_kWT_n$ to the multi-virtual twin group $M_kVT_n$, the welded twin group $WT_n$, and the corresponding virtual and welded braid-type groups. 
Several structural properties of $M_kWT_n$ are obtained. In particular, we compute its abelianization, prove that its commutator subgroup is perfect for $n\ge5$, and show that the symmetric group $S_n$ is its smallest non-abelian finite quotient. 
We also investigate the representation theory of $M_kWT_n$. In fact, we classify all non-trivial  complex homogeneous $2$-local representations of $M_kWT_n$, showing that only one family survives under the additional twin and welded relations. Furthermore, we classify all non-trivial complex homogeneous $3$-local representations of $M_2WT_n$. We further investigate the reducibility and faithfulness properties of both the $2$-local and $3$-local representations.
\end{abstract}

\vspace*{0.75cm}

\section{Introduction}

Since the introduction of virtual knot theory by Kauffman \cite{K}, several generalizations of classical braid groups have appeared in the literature, including virtual braid groups, welded braid groups, twin groups, virtual twin groups, welded twin groups, and their associated pure subgroups. These groups form a rich family of braid-type groups with strong connections to low-dimensional topology, combinatorial group theory, and representation theory.\vspace{0.1cm}

Twin groups were introduced as Coxeter-type analogues of braid groups, obtained by replacing the braid relations with involutivity relations. The twin group $T_n$ was originally defined by Voevodsky and Shabat in~\cite{SV}, and later by Khovanov who investigated the group extensively in~\cite{M.Khovanov1996} \cite{MK}, establishing a geometric framework similar to the one associated with the classical braid group $B_n$. Structural properties of twin and pure twin groups were studied by Bardakov, Singh, and Vesnin in \cite{BSV}. More recently, several authors investigated local and linear representations of braid-type and twin-type groups; see for instance \cite{KMPN}, \cite{Nassertwin}, and \cite{Nasser2026}. \vspace{0.1cm}

A new direction in this area was initiated by the introduction of multi-virtual structures. 
Motivated by multi-virtual knot theory \cite{Kau2}, several multi-virtual and multi-welded braid-type groups have recently been introduced and studied. In \cite{KNP}, Keshari, Nasser, and Prabhakar investigated local representations of the multi-virtual braid group $M_kVB_n$ and the multi-welded braid group $M_kWB_n$. Subsequently, Keshari, Mayassi, Prabhakar, and Nasser introduced the multi-virtual twin group $M_kVT_n$ and studied its presentations and local representations in \cite{KMPN}. In parallel, universal virtual braid groups $UV_n(k)$ were introduced by Ocampo in \cite{O2} as a unified framework interpolating between several virtual braid-type groups.\vspace{0.1cm}

The purpose of the present paper is to introduce and study a welded analogue of the multi-virtual twin group. More precisely, we define the \emph{multi-welded twin group} $M_kWT_n$, which consists of one twin family together with $k$ virtual families interacting through mixed virtual and welded relations. Our construction naturally fits into the universal welded framework developed in \cite{NasserOcampo}.\vspace{0.1cm}

One of the main observations of the paper is that $M_kWT_n$ arises naturally as a quotient of the universal welded braid group $UW_n(k)$. More precisely, one distinguished non-virtual family of $UW_n(k)$ becomes the twin family, while the remaining non-virtual families become additional virtual families. This construction also yields natural quotient maps relating $M_kWT_n$ to $M_kVT_n$, $WT_n$, and the corresponding virtual and welded braid-type groups. In particular, we obtain the description
$$
WT_n \cong \faktor{UW_n(1)}{\langle\!\langle \sigma_{i,1}^2\rangle\!\rangle},
$$
which provides a natural welded counterpart of the quotient description of the group $VT_n$ obtained in \cite{O2}.\vspace{0.1cm}

We also establish several structural properties of $M_kWT_n$. In particular, we compute its abelianization, prove that its commutator subgroup is perfect for $n\ge5$, and show that the symmetric group $S_n$ is its smallest non-abelian finite quotient. These results show that several rigidity phenomena from the universal virtual and welded framework survive after passing to the involutive quotient defining $M_kWT_n$.\vspace{0.1cm}

The second part of the paper is devoted to the representation theory of $M_kWT_n$. Using the classification of homogeneous local representations of $UW_n(k)$ obtained in \cite{NasserOcampo}, we classify all non-trivial complex homogeneous $2$-local representations of $M_kWT_n$. Interestingly, among the three families of homogeneous $2$-local representations of $UW_n(k)$, only one survives after imposing the additional twin and welded relations defining $M_kWT_n$. We also classify all non-trivial complex homogeneous $3$-local representations of $M_2WT_n$. In addition, we study the reducibility and faithfulness properties of both the $2$-local and $3$-local representations.\vspace{0.1cm}

The paper is divided into two sections as follows. In Section~\ref{sec:multiwelded}, we introduce the multi-welded twin group $M_kWT_n$, establish its connections with the universal welded braid group $UW_n(k)$, and prove several structural properties. In Section~\ref{sec:reps}, we investigate homogeneous local representations of $M_kWT_n$, including the classification of homogeneous $2$-local and $3$-local representations, together with the study of their main properties, especially faithfulness and reducibility.

\subsection*{Acknowledgments}

The second named author gratefully acknowledges the personal and medical support of Eliane Santos, the staff of HCA, Bruno Noronha, Luciano Macedo, M\'arcio Isabella, Andreia de Oliveira Rocha, Andreia Gracielle Santana, Ednice de Souza Santos, and SMURB--UFBA (Servi\c{c}o M\'edico Universit\'ario Rubens Brasil Soares), whose support since July 2024 was essential in enabling the completion of this work. O.~O.~was partially supported by the National Council for Scientific and Technological Development (CNPq, Brazil) through a \textit{Bolsa de Produtividade} grant No.~305422/2022--7.

\vspace*{0.1 cm}

\section{The Multi-Welded Twin Group $M_kWT_n$}\label{sec:multiwelded}

The group introduced below should be viewed as the natural welded analogue of the multi-virtual twin group introduced in \cite{KMPN}. In contrast with the universal virtual and welded braid groups $UV_n(k)$ \cite{O2} and $UW_n(k)$ \cite{NasserOcampo}, the group $M_kWT_n$ consists of one twin family together with $k$ virtual families. 
The defining relations are obtained from the multi-virtual twin group by adjoining, for each virtual family, the welded relation involving that family and the twin generators.

\begin{definition}\label{def:MWT}
Let $k \ge 1$ and $n \ge 2$. The \emph{multi-welded twin group} $M_kWT_n$ is the group generated by
$$
s_i\;(1 \le i \le n-1) \quad\text{and}\quad \rho_i^{(\alpha)}\;(1 \le i \le n-1,\; 0 \le \alpha \le k-1)
$$
subject to the following relations for $0 \le \alpha,\, \beta \le k-1$:
\begin{align}
\label{eqn:pres_tinv}
&\text{(Twin involutions)} && s_i^2 = 1;\\ 
\label{eqn:pres_tcr}
&\text{(Twin commutation)} && s_i s_j = s_j s_i\;\;(|i-j|\ge 2); \\ 
\label{eqn:pres_coxeter}
&\text{(Virtual Coxeter)} && (\rho_i^{(\alpha)})^2 = 1,\quad \rho_i^{(\alpha)}\rho_{i+1}^{(\alpha)}\rho_i^{(\alpha)} = \rho_{i+1}^{(\alpha)}\rho_i^{(\alpha)}\rho_{i+1}^{(\alpha)},\\ 
\nonumber 
& && \rho_i^{(\alpha)}\rho_j^{(\alpha)} = \rho_j^{(\alpha)}\rho_i^{(\alpha)}\;\;(|i-j|\ge 2);\\ 
\label{eqn:pres_tmc}
&\text{(Mixed commutation)} && \rho_i^{(\alpha)} s_j = s_j \rho_i^{(\alpha)}\;\;(|i-j|\ge 2);\\
\label{eqn:pres_tmr2type}
&\text{(MR2-type)} && \rho_i^{(\alpha)}\rho_{i+1}^{(\alpha)} s_i = s_{i+1} \rho_i^{(\alpha)}\rho_{i+1}^{(\alpha)}\quad (1\le i\le n-2);\\
\label{eqn:pres_tvc}
&\text{(Virtual commutation)} && \rho_i^{(\alpha)}\rho_j^{(\beta)} = \rho_j^{(\beta)}\rho_i^{(\alpha)}\;\;(|i-j|\ge 2,\; \alpha\neq\beta);\\
\label{eqn:pres_tmr1}
&\text{(Virtual mixed relations I)} && \rho_i^{(\alpha)}\rho_{i+1}^{(\beta)}\rho_i^{(\beta)} = \rho_{i+1}^{(\beta)}\rho_i^{(\beta)}\rho_{i+1}^{(\alpha)}\quad (1\le i\le n-2,\; \alpha<\beta);\\
\label{eqn:pres_tmr2}
&\text{(Virtual mixed relations II)} && \rho_i^{(\alpha)}\rho_{i+1}^{(\alpha)}\rho_i^{(\beta)} = \rho_{i+1}^{(\beta)}\rho_i^{(\alpha)}\rho_{i+1}^{(\alpha)}\quad (1\le i\le n-2,\; \alpha<\beta);\\
\label{eqn:pres_twr}
&\text{(Welded relations)} && \rho_i^{(\alpha)} s_{i+1} s_i = s_{i+1} s_i \rho_{i+1}^{(\alpha)}\quad (1\le i\le n-2). 
\end{align}
For $n=2$ the relations with index $i$ are vacuous, and for $n=1$ the group is trivial.
\end{definition}

\begin{remark}
\begin{enumerate}
    \item For $k=1$ we recover the classical welded twin group 
    $$
    M_1WT_n = WT_n,
    $$ 
    where $WT_n$ denotes the welded twin group introduced in \cite[Section~5]{BSV}. 
In particular, the structural results established later in Subsection~\ref{subsec:properties} apply to the classical welded twin group as a special case. 

\item As in the classical welded twin group $WT_n$, the second forbidden relation also holds in $M_kWT_n$. Indeed, taking inverses in the welded relation \eqref{eqn:pres_twr} and using the involutivity of the generators, we obtain
$$
s_i s_{i+1}\rho_i^{(\alpha)} = \rho_{i+1}^{(\alpha)}s_i s_{i+1}.
$$
Thus, for each virtual family $\alpha$, the two twin-welded forbidden relations are equivalent in $M_kWT_n$.
\end{enumerate} 
\end{remark}

There is a natural epimorphism from the universal welded braid group $UW_n(k)$ defined in \cite[Section~5]{NasserOcampo} onto the multi-welded twin group $M_kWT_n$ given in Definition~\ref{def:MWT}. Under this map, the virtual generators $\rho_i$ of $UW_n(k)$ are identified with the distinguished virtual family $\rho_i^{(0)}$, the first $k-1$ non-virtual families $\sigma_{i,t}$ ($1\le t\le k-1$) become the additional virtual families $\rho_i^{(t)}$, and the last non-virtual family $\sigma_{i,k}$ is sent to the twin generators $s_i$. In this sense, the multi-welded twin group is obtained by distinguishing one non-virtual family of $UW_n(k)$ as a twin family while the remaining non-virtual families become virtual ones.  The precise statement is given in the following proposition.

\begin{proposition}\label{prop:UW_MkWT}
For all $k \ge 1$ and $n \ge 2$, there exists a surjective homomorphism
$$
\Psi_{n,k}\colon UW_n(k) \longrightarrow M_kWT_n
$$
defined on the generators of $UW_n(k)$ by
$$
\Psi_{n,k}(\rho_i) = \rho_i^{(0)}\ \text{ and } \  \Psi_{n,k}(\sigma_{i,t}) = 
\begin{cases}
\rho_i^{(t)}, & \text{if } 1 \le t \le k-1,\\[4pt]
s_i, & \text{if } t = k
\end{cases}
$$
for all $i = 1,\dots,n-1$. The kernel of $\Psi_{n,k}$ contains the normal closure of $\{(\sigma_{i,t})^2 : 1\le i\le n-1,\;1\le t\le k\}$.
\end{proposition}

\begin{proof}
We verify that the images of the generators satisfy the defining relations given in the presentation of $UW_n(k)$ \cite[Definitions~1 and~14]{NasserOcampo}.

First, the relations involving only the virtual generators $\rho_i$, namely (PR1)--(PR3), are preserved because
$$
\Psi_{n,k}(\rho_i)=\rho_i^{(0)}
$$
and the family $\{\rho_i^{(0)}\}_{i=1}^{n-1}$ satisfies the Coxeter relations \eqref{eqn:pres_coxeter} in $M_kWT_n$.

Next, consider the commutation relation (CR). Let $|i-j|\geq 2$ and $t,\ell\in\{1,\ldots,k\}$. If $t,\ell<k$, then 
$$
\Psi_{n,k}(\sigma_{i,t})=\rho_i^{(t)}\ \text{ and }
\ \Psi_{n,k}(\sigma_{j,\ell})=\rho_j^{(\ell)},
$$
and these elements commute by the virtual commutation relations \eqref{eqn:pres_tvc}. If $t=\ell=k$, then the images are $s_i$ and $s_j$, which commute by the twin commutation relations \eqref{eqn:pres_tcr}. Finally, if exactly one of $t,\ell$ is equal to $k$, then the images are one twin generator and one virtual generator with distant indices, and they commute by the mixed commutation relations \eqref{eqn:pres_tmc}. Hence (CR) is preserved. 

The verification of (MR1) is analogous. If $|i-j|\geq 2$, then $\Psi_{n,k}(\rho_j)=\rho_j^{(0)}$. For $t<k$, the elements $\rho_i^{(t)}$ and $\rho_j^{(0)}$ commute by the virtual commutation relations \eqref{eqn:pres_tvc}, whereas for $t=k$, the elements $s_i$ and $\rho_j^{(0)}$ commute by the mixed commutation relations \eqref{eqn:pres_tmc}. Therefore (MR1) is preserved.

It remains to check the two relations involving adjacent indices. We first consider (MR2). For $1\leq i\leq n-2$ and $1\leq t\leq k$, we have 
$$
\Psi_{n,k}(\rho_i\rho_{i+1}\sigma_{i,t}) = \rho_i^{(0)}\rho_{i+1}^{(0)}\Psi_{n,k}(\sigma_{i,t}),
$$
whereas
$$
\Psi_{n,k}(\sigma_{i+1,t}\rho_i\rho_{i+1}) = \Psi_{n,k}(\sigma_{i+1,t})\rho_i^{(0)}\rho_{i+1}^{(0)}.
$$
If $t<k$, this becomes
$$
\rho_i^{(0)}\rho_{i+1}^{(0)}\rho_i^{(t)} = \rho_{i+1}^{(t)}\rho_i^{(0)}\rho_{i+1}^{(0)},
$$
which is one of the virtual mixed relations \eqref{eqn:pres_tmr2} in Definition~\ref{def:MWT}. If $t=k$, it becomes 
$$
\rho_i^{(0)}\rho_{i+1}^{(0)}s_i = s_{i+1}\rho_i^{(0)}\rho_{i+1}^{(0)},
$$
which is precisely the MR2-type relation \eqref{eqn:pres_tmr2type}. Thus (MR2) is preserved.

We now check the welded relation (WR1) of $UW_n(k)$. For $1\leq i\leq n-2$ and $1\leq t\leq k$, its image under $\Psi_{n,k}$ is the equality 
$$
\rho_i^{(0)}\Psi_{n,k}(\sigma_{i+1,t})\Psi_{n,k}(\sigma_{i,t}) = \Psi_{n,k}(\sigma_{i+1,t})\Psi_{n,k}(\sigma_{i,t})\rho_{i+1}^{(0)}.
$$
If $t<k$, this becomes
$$
\rho_i^{(0)}\rho_{i+1}^{(t)}\rho_i^{(t)} = \rho_{i+1}^{(t)}\rho_i^{(t)}\rho_{i+1}^{(0)},
$$
which is one of the virtual mixed relations \eqref{eqn:pres_tmr1} in Definition~\ref{def:MWT}. If $t=k$, it becomes 
$$
\rho_i^{(0)}s_{i+1}s_i = s_{i+1}s_i\rho_{i+1}^{(0)},
$$
which is exactly the welded relation \eqref{eqn:pres_twr} defining $M_kWT_n$. Hence (WR1) is also preserved.

Therefore $\Psi_{n,k}$ is a well-defined homomorphism. It is surjective because all generators of $M_kWT_n$ lie in its image:
$$
s_i=\Psi_{n,k}(\sigma_{i,k}),\qquad
\rho_i^{(0)}=\Psi_{n,k}(\rho_i),\ \text{ and } \ 
\rho_i^{(t)}=\Psi_{n,k}(\sigma_{i,t})
\quad (1\leq t\leq k-1).
$$

Finally, we prove the statement about the kernel. For $t=k$,
$$
\Psi_{n,k}(\sigma_{i,k}^2)=s_i^2=1,
$$
and for $1\leq t\leq k-1$,
$$
\Psi_{n,k}(\sigma_{i,t}^2)=(\rho_i^{(t)})^2=1.
$$
Thus $(\sigma_{i,t})^2\in\ker(\Psi_{n,k})$ for every $i$ and $t$. Since $\ker(\Psi_{n,k})$ is normal, it contains the normal closure 
$$
\left\langle\!\left\langle
(\sigma_{i,t})^2
\mid
1\leq i\leq n-1,\ 1\leq t\leq k
\right\rangle\!\right\rangle .
$$
\end{proof}

Proposition~\ref{prop:UW_MkWT} shows that the multi-welded twin group $M_kWT_n$ is naturally controlled by the universal welded braid group $UW_n(k)$. However, unlike the case $k=1$, where no additional virtual families occur, the quotient is not obtained merely by imposing involutivity on the non-virtual generators. The additional multi-virtual relations in Definition~\ref{def:MWT} play an essential role in encoding the interaction between the different virtual families.

An analogous construction can be performed at the level of the universal virtual braid group $UV_n(k)$. In this case, the resulting quotient is naturally related to the multi-virtual twin group $M_kVT_n$, defined in \cite[Section~3]{KMPN}. However, in contrast with the welded setting, the quotient by the squares of the generators is not sufficient in general to recover all defining relations of $M_kVT_n$.

\begin{proposition}\label{prop:UV_MkVT}
For all $k \ge 1$ and $n \ge 2$, there exists a surjective homomorphism $\Phi_{n,k}\colon UV_n(k) \longrightarrow M_kVT_n$ defined on the generators of $UV_n(k)$ by
$$
\Phi_{n,k}(\rho_i) = \rho_i^{(0)} \ \text{ and } \ 
\Phi_{n,k}(\sigma_{i,t}) = 
\begin{cases}
\rho_i^{(t)}, & \text{if } 1 \le t \le k-1,\\[4pt]
s_i, & \text{if } t = k
\end{cases}
$$
for all $i = 1,\dots,n-1$. The kernel of $\Phi_{n,k}$ contains the normal closure of $\{(\sigma_{i,t})^2 : 1\le i\le n-1,\;1\le t\le k\}$. 
\end{proposition}

\begin{proof}
We verify that the images of the generators satisfy the defining relations of
$UV_n(k)$.

The relations (PR1)--(PR3) are preserved because $\Phi_{n,k}(\rho_i)=\rho_i^{(0)}$ and the family $\{\rho_i^{(0)}\}_{i=1}^{n-1}$ satisfies the Coxeter relations in $M_kVT_n$.

We next check the commutation relation (CR). Let $|i-j|\geq 2$ and $t,\ell\in\{1,\ldots,k\}$. If $t,\ell<k$, then the images are $\rho_i^{(t)}$ and $\rho_j^{(\ell)}$, which commute by the virtual commutation relations in $M_kVT_n$. If $t=\ell=k$, then the images are $s_i$ and $s_j$, which commute by the twin commutation relations. Finally, if exactly one of $t,\ell$ is equal to $k$, then the images are one twin generator and one virtual generator with distant indices, and they commute by the mixed commutation relations. Hence (CR) is preserved.

The relation (MR1) is verified in the same way. Indeed, for $|i-j|\geq 2$, the element $\Phi_{n,k}(\rho_j)$ is $\rho_j^{(0)}$. If $t<k$, then $\Phi_{n,k}(\sigma_{i,t})=\rho_i^{(t)}$, and the required commutation follows from the virtual commutation relations. If $t=k$, then $\Phi_{n,k}(\sigma_{i,k})=s_i$, and the required commutation follows from the mixed commutation relations.

It remains to check (MR2). For $1\leq i\leq n-2$ and $1\leq t\leq k$, we have
$$
\Phi_{n,k}(\rho_i\rho_{i+1}\sigma_{i,t}) = \rho_i^{(0)}\rho_{i+1}^{(0)}\Phi_{n,k}(\sigma_{i,t}),
$$
while
$$
\Phi_{n,k}(\sigma_{i+1,t}\rho_i\rho_{i+1}) = \Phi_{n,k}(\sigma_{i+1,t})\rho_i^{(0)}\rho_{i+1}^{(0)}.
$$
If $t<k$, this becomes
$$
\rho_i^{(0)}\rho_{i+1}^{(0)}\rho_i^{(t)} = \rho_{i+1}^{(t)}\rho_i^{(0)}\rho_{i+1}^{(0)},
$$
which is precisely relation~(18) in \cite{KMPN}, with $\alpha=0$ and $\beta=t$. If $t=k$, it becomes
$$
\rho_i^{(0)}\rho_{i+1}^{(0)}s_i = s_{i+1}\rho_i^{(0)}\rho_{i+1}^{(0)},
$$
which is relation~(19) in \cite{KMPN}, again with $\alpha=0$. Thus (MR2) is preserved.

Therefore $\Phi_{n,k}$ is a well-defined homomorphism. It is surjective because all generators of $M_kVT_n$ lie in its image:
$$
s_i=\Phi_{n,k}(\sigma_{i,k}),\qquad \rho_i^{(0)}=\Phi_{n,k}(\rho_i), \ \text{ and } \ 
\rho_i^{(t)}=\Phi_{n,k}(\sigma_{i,t}) \quad (1\leq t\leq k-1).
$$

Finally, we prove the statement about the kernel. If $t=k$, then
$$
\Phi_{n,k}(\sigma_{i,k}^2)=s_i^2=1
$$
by the twin involution relation in $M_kVT_n$. If $1\leq t\leq k-1$, then
$$
\Phi_{n,k}(\sigma_{i,t}^2)=(\rho_i^{(t)})^2=1
$$
by the virtual involution relations. Hence $(\sigma_{i,t})^2\in\ker(\Phi_{n,k})$ for every $i$ and $t$. Since $\ker(\Phi_{n,k})$ is normal, it contains the normal closure
$$
\left\langle\!\left\langle
(\sigma_{i,t})^2
\mid
1\leq i\leq n-1,\ 1\leq t\leq k
\right\rangle\!\right\rangle .
$$
Consequently, $\Phi_{n,k}$ factors through the quotient by these squares, yielding the induced surjective homomorphism $\overline{\Phi}_{n,k}$.
\end{proof}

We notice that $\Phi_{n,k}$ of Proposition~\ref{prop:UV_MkVT} induces a surjective homomorphism
$$
\overline{\Phi}_{n,k}\colon \faktor{UV_n(k)}{\langle\!\langle (\sigma_{i,t})^2 \rangle\!\rangle} \longrightarrow M_kVT_n.
$$
In general this quotient is not an isomorphism, since further relations are needed to recover the full set of defining relations of $M_kVT_n$.

\begin{corollary}\label{cor:MWT_quotient}
Let $n\geq 3$ and $k\geq 1$. 
The group $M_kWT_n$ is a quotient of $M_kVT_n$. More precisely, $M_kWT_n \cong
\faktor{M_kVT_n} {N}$, where 
$$
N = \left\langle\!\left\langle \rho_i^{(\alpha)}s_{i+1}s_i (s_{i+1}s_i\rho_{i+1}^{(\alpha)})^{-1} \;\middle|\; 1\le i\le n-2,\;0\le \alpha\le k-1 \right\rangle\!\right\rangle.
$$ 
\end{corollary}

\begin{proof}
The presentation of $M_kWT_n$ contains all defining relations of $M_kVT_n$, together with the additional welded relations
$$
\rho_i^{(\alpha)}s_{i+1}s_i = s_{i+1}s_i\rho_{i+1}^{(\alpha)}
$$
for $1\le i\le n-2$ and $0\le \alpha\le k-1$.  Hence the identity map on the generators of $M_kVT_n$ induces a surjective homomorphism from the displayed quotient onto $M_kWT_n$.  Conversely, the defining presentation of the quotient is exactly the presentation of Definition~\ref{def:MWT}.  Therefore the two groups are isomorphic.
\end{proof}

The following commutative diagram summarizes the relationships among the various virtual and welded braid-type and twin-type groups, the latters discussed in this paper. All arrows are surjective homomorphisms, extending the unified framework established in \cite{NasserOcampo} and \cite{O2}.
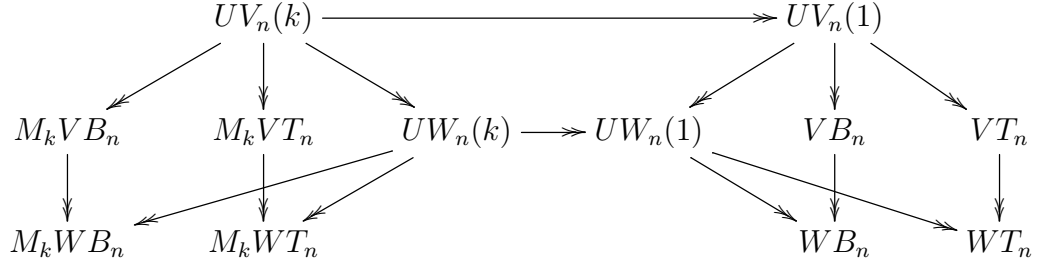
$$ 
    \xymatrix@C=2.0pc{
 & UV_n(k) \ar@{->>}[ld] \ar@{->>}[d] \ar@{->>}[rd] \ar@{->>}[rrr] & & & UV_n(1) \ar@{->>}[ld] \ar@{->>}[d] \ar@{->>}[rd]  & \\
M_kVB_n \ar@{->>}[d] & M_kVT_n \ar@{->>}[d] & UW_n(k) \ar@{->>}[lld] \ar@{->>}[ld] \ar@{->>}[r] & UW_n(1) \ar@{->>}[rd] \ar@{->>}[rrd] & VB_n \ar@{->>}[d] & VT_n \ar@{->>}[d] \\
M_kWB_n & M_kWT_n & & & WB_n & WT_n
}
$$

The top row consists of universal virtual braid groups \cite{O2}. The left column corresponds to multi-version groups (with $k$ families of virtual generators), while the right column corresponds to the classical case $k=1$. Moving downwards, the arrows are natural quotient projections that correspond to imposing additional relations.  
For instance, the quotient map from $UW_n(k)$ to $M_kWT_n$ is the one described in Proposition~\ref{prop:UW_MkWT}; it imposes the involutivity of the images of the non-virtual generators together with the additional multi-virtual relations encoded in Definition~\ref{def:MWT}. 
The existence of all these quotients follows from \cite[Proposition~2.2]{O2} (for the virtual braid-type groups), from Proposition~\ref{prop:UW_MkWT} (for $M_kWT_n$), and from the corresponding results in \cite{NasserOcampo} (for the welded braid-type groups).

The case $k=1$ deserves special attention. It recovers the usual welded twin group and shows that $WT_n$ itself arises naturally from the universal welded framework by imposing the Coxeter-type involutivity relation on the single family of non-virtual generators. 
For this case, the map $UW_n(1)\twoheadrightarrow M_1WT_n$ sends $\rho_i$ to $\rho_i^{(0)}$ and $\sigma_{i,1}$ to $s_i$. Since there are no additional virtual families when $k=1$, the only extra relation imposed on the non-virtual family is the involutivity relation $\sigma_{i,1}^2=1$.  Hence $M_1WT_n$ recovers the usual welded twin group $WT_n$.

\begin{remark}\label{rem:WT_UW}
In the special case $k=1$, Corollary~\ref{cor:MWT_quotient} recovers the usual welded twin group $WT_n$. Equivalently, the natural quotient map $UW_n(1)\twoheadrightarrow WT_n$ is obtained by imposing $\sigma_{i,1}^2=1$ on the unique non-virtual family. 
Thus
$$
WT_n \cong \faktor{UW_n(1)}{\langle\!\langle \sigma_{i,1}^2\rangle\!\rangle}.
$$
This complements the description of $VT_n$ as a quotient of $UV_n(1)$ given in \cite[Proposition~2.2(iii)]{O2}.
\end{remark}

This observation complements the classical quotient description
$$
UV_n(1)\twoheadrightarrow VT_n\twoheadrightarrow WT_n.
$$
Indeed, the universal welded group $UW_n(1)$ provides a direct intermediate framework in which the welded twin group is obtained simply by imposing the involutivity of the unique non-virtual family. This gives a uniform interpretation of the virtual and welded twin groups within the universal framework.

\subsection{Structural Properties Inherited From $UW_n(k)$}\label{subsec:properties}

We now record some structural consequences of the quotient maps constructed above.  Although $M_kWT_n$ is not obtained from $UW_n(k)$ merely by imposing the square relations $(\sigma_{i,t})^2=1$ when $k>1$, it is still a quotient of $UV_n(k)$.  Therefore several properties controlled by finite quotients and commutator subgroups can be transferred from $UV_n(k)$.  The abelianization, on the other hand, is computed directly from the presentation of Definition~\ref{def:MWT}.

\begin{theorem}\label{thm:MWT_properties}
Let $k\ge 1$. Then:
\begin{enumerate}
\item[(i)] For every $n\ge 2$, the abelianization is
$$
(M_kWT_n)^{\mathrm{ab}} \cong \mathbb{Z}_2^{\,k+1}.
$$

\item[(ii)] If $n \ge 5$, then the commutator subgroup $(M_kWT_n)'$ is perfect.

\item[(iii)] If $n \ge 5$, then the symmetric group $S_n$ is the smallest non-abelian finite quotient of $M_kWT_n$. More precisely, every non-abelian finite quotient of $M_kWT_n$ has order at least $n!$, and equality occurs only for $S_n$.

\item[(iv)] For every $n\ge 2$, the group $M_kWT_n$ admits a finite quotient whose order is at least $2^k n!$, hence strictly larger than $n!$.
\end{enumerate}
\end{theorem}

\begin{proof}
We use the presentation of $M_kWT_n$ given in Definition~\ref{def:MWT}, with generators
$$
s_i \quad (1\le i\le n-1) \ \text{ and } \ 
\rho_i^{(\alpha)} \quad (1\le i\le n-1,\;0\le \alpha\le k-1).
$$

\begin{enumerate}
\item[(i)] In the abelianization all generators commute. Each generator is an involution, hence has order dividing $2$. The mixed relation 
$$
\rho_i^{(\alpha)}\rho_{i+1}^{(\alpha)} s_i = s_{i+1} \rho_i^{(\alpha)}\rho_{i+1}^{(\alpha)}
$$
implies, in the abelianization, $s_i = s_{i+1}$ for all $i$ (since the $\rho$ terms cancel because they are involutions and commute in the abelianization). Thus all $s_i$ have the same class; denote it by $\bar s$.

For the virtual generators, the commutation relations $\rho_i^{(\alpha)}\rho_j^{(\beta)} = \rho_j^{(\beta)}\rho_i^{(\alpha)}$ for $|i-j|\ge 2$ together with the Coxeter relations imply that, in the abelianization, all $\rho_i^{(\alpha)}$ with the same superscript $\alpha$ are identified, but different superscripts are not identified with each other (no relation forces $\rho_i^{(\alpha)} = \rho_i^{(\beta)}$ for $\alpha \neq \beta$). Hence, for each $\alpha = 0,\dots,k-1$, we obtain a distinct class $\bar\rho^{(\alpha)}$ of order $2$.

Thus $(M_kWT_n)^{\mathrm{ab}}$ is generated by $\bar s, \bar\rho^{(0)}, \dots, \bar\rho^{(k-1)}$, all of order $2$, with no further relations. The map $M_kWT_n \to \mathbb{Z}_2^{\,k+1}$ defined by
$$
s_i \mapsto e_0 \ \text{ and } \ 
\rho_i^{(\alpha)} \mapsto e_{\alpha+1}
$$
(where $e_0,\dots,e_k$ are the standard basis vectors) satisfies all defining relations and is surjective. Hence $(M_kWT_n)^{\mathrm{ab}} \cong \mathbb{Z}_2^{\,k+1}$.

\item[(ii)] Since $M_kWT_n$ is a quotient of $UW_n(k)$ (Proposition~\ref{prop:UW_MkWT}), and $UW_n(k)$ is a quotient of $UV_n(k)$, there exists a surjective homomorphism $UV_n(k) \twoheadrightarrow M_kWT_n$. For any surjective homomorphism $f\colon G\twoheadrightarrow H$, one has $f(G') = H'$. Hence the map induces a surjection $(UV_n(k))' \twoheadrightarrow (M_kWT_n)'$. By \cite[Proposition~3.2]{O2}, the commutator subgroup $(UV_n(k))'$ is perfect for $n\ge 5$. Since a homomorphic image of a perfect group is perfect, it follows that $(M_kWT_n)'$ is perfect.

\item[(iii)] Let $\psi\colon M_kWT_n \twoheadrightarrow G$ be a finite non-abelian quotient. Composing with the natural projection $UV_n(k) \twoheadrightarrow M_kWT_n$ yields a finite non-abelian quotient of $UV_n(k)$. By \cite[Theorem~5.5]{O2}, every finite non-abelian quotient of $UV_n(k)$ has order at least $n!$. Hence $|G| \ge n!$. Moreover, if $|G| = n!$, then the same result implies that $G \cong S_n$, up to the exceptional outer automorphism when $n=6$. Thus $S_n$ is the smallest non-abelian finite quotient of $M_kWT_n$.

To see that $S_n$ is indeed a quotient, define $\pi\colon M_kWT_n \to S_n$ by
$$
\pi(\rho_i^{(\alpha)}) = (i\ i+1) \ \text{ and } \ 
\pi(s_i) = (i\ i+1)
$$
for all $i$ and $\alpha$. The defining relations of $M_kWT_n$ are sent to valid relations in $S_n$: the twin and virtual Coxeter relations follow from the standard Coxeter presentation of $S_n$, and the mixed relations (MR2-type and welded) become consequences of the braid relation
$$
(i\ i+1)(i+1\ i+2)(i\ i+1) = (i+1\ i+2)(i\ i+1)(i+1\ i+2).
$$
Hence $\pi$ is a well-defined surjection.

\item[(iv)] Let $A = \mathbb{Z}_2^{\,k+1}$ with basis $e_0, e_1, \dots, e_k$. Consider the finite group $S_n \times A$. Define a homomorphism $\Theta\colon M_kWT_n \to S_n \times A$ on generators by
$$
\Theta(s_i) = ((i\ i+1), e_0) \ \text{ and } \ 
\Theta(\rho_i^{(\alpha)}) = ((i\ i+1), e_{\alpha+1})
$$
for all $i$ and $\alpha = 0,\dots,k-1$. The defining relations of $M_kWT_n$ are satisfied: their first components hold in $S_n$ (as in item (iii)), and their second components hold because $A$ is elementary abelian of exponent $2$.

The image of $\Theta$ contains $((i\ i+1), e_0)$ and $((i\ i+1), e_{\alpha+1})$ for each $\alpha$. Hence it contains
$$
(1, e_0 + e_{\alpha+1}) = ((i\ i+1), e_0)\,((i\ i+1), e_{\alpha+1})
$$
for each $\alpha = 0,\dots,k-1$. Therefore the image contains the subgroup
$$
\langle e_0 + e_1, \dots, e_0 + e_k \rangle \cong \mathbb{Z}_2^k
$$
in the second factor (note that $e_0 + e_{\alpha+1}$ for $\alpha = 0,\dots,k-1$ gives exactly $k$ independent elements). The image also projects onto $S_n$ via the first factor. Consequently, the image is a finite group of order at least $2^k n!$. Thus $M_kWT_n$ admits a finite quotient whose order is at least $2^k n!$, which is strictly larger than $n!$ since $k\ge 1$.
\end{enumerate}
\end{proof}

\begin{remark}
The results of this section show that the multi-welded twin group fits naturally into the universal framework developed in \cite{NasserOcampo} and \cite{O2}. The explicit description $WT_n \cong \faktor{UW_n(1)}{\langle\!\langle \sigma_{i,1}^2\rangle\!\rangle}$ (Remark~\ref{rem:WT_UW}) is, to our knowledge, new and illustrates the unifying power of the universal welded braid group.
\end{remark}

\section{Connection With Representation Theory}\label{sec:reps}

The structural description above also suggests natural questions in representation theory.  Since $M_kWT_n$ is a quotient of $UW_n(k)$, every representation of $M_kWT_n$ gives, by composition, a representation of $UW_n(k)$ satisfying the relations imposed in Definition~\ref{def:MWT}.  Thus one possible strategy is to start from representations of $UW_n(k)$ and impose the additional involutive and multi-virtual twin-type relations required for them to factor through $M_kWT_n$. 
This provides a bridge between the representation theory of universal welded braid groups and the representation theory of multi-virtual twin groups.

The classification of complex homogeneous $2$-local representations of $UW_n(k)$ was obtained in \cite[Theorem~26, Lemma~27]{NasserOcampo}.  A natural next problem is to determine which of those representations factor through $M_kWT_n$.  This requires imposing the relations
$$
s_i^2=1\ \text{ and } \
(\rho_i^{(\alpha)})^2=1,
$$
together with the mixed multi-virtual and welded relations appearing in Definition~\ref{def:MWT}.   This comparison should produce a representation theory that is closely related to, but more restrictive than, the one obtained for $M_kVT_n$ in \cite[Theorem~4.1]{KMPN}. First, we recall the results of all classifications of complex homogeneous $2$-local representations of $UW_n(k)$ for $k\geq 1$ and $n\geq 3$.

\begin{theorem} \cite[Theorem~26, Lemma~27]{NasserOcampo} \label{localUW}
Let $\omega: UW_n(k) \longrightarrow \mathrm{GL}_n(\mathbb{C})$ be a non-trivial homogeneous $2$-local representation. Then, up to equivalence of representations, $\omega$ is equivalent to one of the following three representations, denoted by $\omega_j$ for $1\leq j \leq 3$, explicitly determined as follows.
\begin{enumerate}
\item  The matrices of the representation $\omega_1$ are given by
    \[
\omega_1(\rho_i)=
\begin{pmatrix}
I_{i-1} & 0 & 0 \\
0 & \begin{pmatrix} 0 & 1 \\ 1 & 0 \end{pmatrix} & 0 \\
0 & 0 & I_{n-i-1}
\end{pmatrix}
\text{ \ and \ }
\omega_1(\sigma_{i,t})=
\begin{pmatrix}
I_{i-1} & 0 & 0 \\
0 & \begin{pmatrix} 0 & \dfrac{s_{2,t}}{r_2} \\ 
r_2 s_{3,t} & 0\end{pmatrix} & 0 \\
0 & 0 & I_{n-i-1}
\end{pmatrix}
\]
for all $1 \leq i \leq n-1$ and $1 \leq t \leq k$, where $r_2, s_{2,t}, s_{3,t} \in \mathbb{C}^*$.\vspace{0.1cm}

\item  The matrices of the representation $\omega_2$ are given by
    \[
\omega_2(\rho_i)=
\begin{pmatrix}
I_{i-1} & 0 & 0 \\
0 & \begin{pmatrix} 0 & 1 \\ 1 & 0 \end{pmatrix} & 0 \\
0 & 0 & I_{n-i-1}
\end{pmatrix}
\text{ \ and \ }
\omega_2(\sigma_{i,t})=
\begin{pmatrix}
I_{i-1} & 0 & 0 \\
0 & \begin{pmatrix} 0 & \dfrac{s_{2,t}}{r_2} \\ 
1 & s_{4,t}\end{pmatrix} & 0 \\
0 & 0 & I_{n-i-1}
\end{pmatrix}
\]
for all $1 \leq i \leq n-1$ and $1 \leq t \leq k$, where $r_2,s_{2,t}, s_{4,t} \in \mathbb{C}^*$.\vspace{0.1cm}

\item  The matrices of the representation $\omega_3$ are given by
    \[
\omega_3(\rho_i)=
\begin{pmatrix}
I_{i-1} & 0 & 0 \\
0 & \begin{pmatrix} 0 & 1 \\ 1 & 0 \end{pmatrix} & 0 \\
0 & 0 & I_{n-i-1}
\end{pmatrix}
\text{ \ and \ }
\omega_3(\sigma_{i,t})=
\begin{pmatrix}
I_{i-1} & 0 & 0 \\
0 & \begin{pmatrix} s_{1,t} & \dfrac{s_{2,t}}{r_2} \\ 
1 & 0\end{pmatrix} & 0 \\
0 & 0 & I_{n-i-1}
\end{pmatrix}
\]
for all $1 \leq i \leq n-1$ and $1 \leq t \leq k$, where $r_2,s_{1,t}, s_{2,t} \in \mathbb{C}^*$.
\end{enumerate}
\end{theorem}

We now determine all non-trivial complex homogeneous $2$-local representations of $M_kWT_n$ for all $k\geq 1$ and $n\geq 3$.

\begin{theorem} \label{thmlocal}
Let $\delta: M_kWT_n \longrightarrow \mathrm{GL}_n(\mathbb{C})$ be a non-trivial homogeneous $2$-local representation. Then, up to equivalence of representations, $\delta$ is equivalent to the following representation, denoted by $\delta_1$, explicitly determined as follows.
    \[
\delta_1(\rho_i^{(0)})=
\begin{pmatrix}
I_{i-1} & 0 & 0 \\
0 & \begin{pmatrix} 0 & 1 \\ 1 & 0 \end{pmatrix} & 0 \\
0 & 0 & I_{n-i-1}
\end{pmatrix},\]
$$\delta_1(\rho_i^{(t)})=
\begin{pmatrix}
I_{i-1} & 0 & 0 \\
0 & \begin{pmatrix} 0 & \dfrac{s_{2,t}}{r_2} \\ 
\dfrac{r_2}{s_{2,t}} & 0\end{pmatrix} & 0 \\
0 & 0 & I_{n-i-1}
\end{pmatrix},
$$
and
$$\delta_1(s_i)=
\begin{pmatrix}
I_{i-1} & 0 & 0 \\
0 & \begin{pmatrix} 0 & \dfrac{s}{r_2} \\ 
\dfrac{r_2}{s} & 0\end{pmatrix} & 0 \\
0 & 0 & I_{n-i-1}
\end{pmatrix}
$$
for all $1 \leq i \leq n-1$ and $1 \leq t \leq k-1$, where $r_2, s_{2,t},s \in \mathbb{C}^*$.
\end{theorem}
\begin{proof}
To establish this result, we first consider the representation $\omega$ introduced in Theorem \ref{localUW}. By Proposition \ref{prop:UW_MkWT}, the representation $\delta$ is defined by
\[
\delta(\rho_i^{(0)})=\omega(\rho_i),
\]
\[
\delta(\rho_i^{(t)})=\omega(\sigma_{i,t}), \quad 1\leq t\leq k-1,
\]
and
\[
\delta(s_i)=\omega(\sigma_{i,k}),
\]
subject to additional relations of $M_kWT_n$. Hence, we impose the required relations on each of the three possible forms of $\omega$ given in Theorem \ref{localUW}, namely $\omega_j$ for $1\leq j\leq 3$. The relations that must be considered here are  \eqref{eqn:pres_tinv}, \eqref{eqn:pres_coxeter}, \eqref{eqn:pres_tvc}, \eqref{eqn:pres_tmr1}, \eqref{eqn:pres_tmr2}, and \eqref{eqn:pres_twr}. We now consider the possible cases of $\omega$ separately.
\begin{enumerate}
    \item Suppose that $\omega$ is of the form $\omega_1$. Imposing relations \eqref{eqn:pres_tinv} and \eqref{eqn:pres_coxeter} gives
    \[
\delta(\rho_i^{(t)})^2=\delta(s_i)^2=I_n, \quad 1\leq t\leq k.
    \]
    Equivalently,
    \[
    \omega_1(\sigma_{i,t})^2=I_n, \quad 1\leq t\leq k.
    \]
    A direct computation shows that this condition is equivalent to
    \[
    s_{2,t}=\frac{1}{s_{3,t}}, \quad 1\leq t\leq k.
    \]
    Moreover, straightforward matrix computations show that relations {\eqref{eqn:pres_tvc}, \eqref{eqn:pres_tmr1}, \eqref{eqn:pres_tmr2}, and \eqref{eqn:pres_twr} } are automatically satisfied under these conditions. Therefore, in this case, the representation $\delta$ is equivalent to $\delta_1$, as required.

    \item Now suppose that $\omega$ is of the form $\omega_2$. The case where $\omega$ is of the form $\omega_3$ is completely analogous. As in the previous case, imposing relations \eqref{eqn:pres_tinv} and \eqref{eqn:pres_coxeter} yields
    \[
    \delta(\rho_i^{(t)})^2=\delta(s_i)^2=I_n, \quad 1\leq t\leq k.
    \]
    Equivalently,
    \[
    \omega_2(\sigma_{i,t})^2=I_n, \quad 1\leq t\leq k.
    \]
    Direct computations show that this condition is equivalent to
    \[
    s_{2,t}=r_2
    \quad \text{and} \quad
    s_{4,t}=0,
    \quad 1\leq t\leq k.
    \]
    However, this contradicts the fact that $s_{4,t}\in \mathbb{C}^*$. Hence, the representations $\omega_2$ and $\omega_3$ cannot induce representations of $M_kWT_n$.
\end{enumerate}

Therefore, the representation $\delta$ admits only one possible form, namely $\delta_1$, as required.
\end{proof}

We now state a result regarding the reducibility of these classified representations.

\begin{proposition}
    Consider the representation $\delta_1: M_kWT_n \longrightarrow \mathrm{GL}_n(\mathbb{C})$ determined in Theorem \ref{thmlocal}. Then $\delta_1$ is reducible if and only if $r_2=s_{2,t}=s$ for all $1\leq t \leq k-1$.
\end{proposition}

\begin{proof}
By \cite[Theorem 9]{NasserOcampo}, a representation whose matrices are of the form
\[
\begin{pmatrix}
I_{i-1} & 0 & 0 \\
0 & \begin{pmatrix} 0 & 1 \\ 1 & 0 \end{pmatrix} & 0 \\
0 & 0 & I_{n-i-1}
\end{pmatrix}
\]
together with matrices of the form
\[
\begin{pmatrix}
I_{i-1} & 0 & 0 \\
0 & \begin{pmatrix} a_t & b_t \\ c_t & d_t \end{pmatrix} & 0 \\
0 & 0 & I_{n-i-1}
\end{pmatrix}
\]
is reducible if and only if either
\[
a_t+b_t=1 \quad \text{and} \quad c_t+d_t=1
\]
for all \(t\), or
\[
a_t+c_t=1 \quad \text{and} \quad b_t+d_t=1
\]
for all \(t\). Applying this criterion to our representation yields the desired conclusion.
\end{proof}

We now state a result regarding the faithfulness of these classified representations.

\begin{proposition}
The representation $\delta_1: M_kWT_n \longrightarrow \mathrm{GL}_n(\mathbb{C})$ determined in Theorem \ref{thmlocal} is unfaithful.
\end{proposition}

\begin{proof}
Using matrix computations, we obtain
\[
[\delta_1(\rho_i^{(0)})\delta_1(s_{i+1})]^3=I_n,
\quad 1\leq i \leq n-2.
\]
Notice that the element $(\rho_i^{(0)}s_{i+1})^3$ is non-trivial. Indeed, if
\[
(\rho_i^{(0)}s_{i+1})^3=1,
\]
then we would obtain
\[
\rho_i^{(0)}s_{i+1}\rho_i^{(0)}
=
s_{i+1}\rho_i^{(0)}s_{i+1},
\]
which is not a valid relation in our group. Therefore, the representation $\delta_1$ is unfaithful.
\end{proof}

It would also be interesting to extend the $3$-local representation theory from $UV_n(2)$ (see \cite[Section~4]{NasserOcampo}) to the multi-welded twin setting. This could lead to new families of representations and further insights into the structure of $M_kWT_n$. We now determine all non-trivial complex homogeneous $3$-local complex representation of $M_2WT_n$ for all $n\geq 4$.

\begin{theorem} \label{thm3localM2}
Let $\epsilon: M_2WT_n \longrightarrow \mathrm{GL}_{n+1}(\mathbb{C})$ be a non-trivial homogeneous $3$-local representation. Then, up to equivalence of representations, $\epsilon$ is equivalent to one of the following four representations, denoted by $\epsilon_j$, $1 \leq j \leq 4$, given by
$$
\epsilon_j(s_i)=
\begin{pmatrix}
I_{i-1} & 0 & 0 \\
0 & S^{(j)} & 0 \\
0 & 0 & I_{n-i-1}
\end{pmatrix},$$
$$\ 
\epsilon_j(\rho_i^{(0)})=
\begin{pmatrix}
I_{i-1} & 0 & 0 \\
0 & R_{0}^{(j)} & 0 \\
0 & 0 & I_{n-i-1}
\end{pmatrix}, 
$$
and
\[ 
\epsilon_j(\rho_i^{(1)})=
\begin{pmatrix}
I_{i-1} & 0 & 0 \\
0 & R_{1}^{(j)} & 0 \\
0 & 0 & I_{n-i-1}
\end{pmatrix}
\]
for all $1 \leq i \leq n-1$, where $S^{(j)}$, $R_{0}^{(j)}$, and $R_{1}^{(j)}$ are explicitly determined as follows.
\begin{enumerate}
    \item The matrices of the representation $\epsilon_1$ are given by
    $$S^{(1)}=\begin{pmatrix}
        1 & 0 & 0\\
        0 & 0 & \dfrac{1}{s_8}\\
        0 & s_8 & 0
    \end{pmatrix}, \ R_{0}^{(1)}=\begin{pmatrix}
        1 & 0 & 0\\
        0 & 0 & r_6\\
        0 & \dfrac{1}{r_6} & 0
    \end{pmatrix}, \text{ and }R_{1}^{(1)}=\begin{pmatrix}
         1 & 0 & 0\\
        0 & 0 & p_6\\
        0 & \dfrac{1}{p_6} & 0
    \end{pmatrix},
    $$
    where $s_8,r_6,p_6 \in \mathbb{C}^*$.\vspace{0.1cm}
     \item The matrices of the representation $\epsilon_2$ are given by
    $$
    S^{(2)}=\begin{pmatrix}
        0 & \dfrac{1}{s_4} & 0\\
        s_4 & 0 & 0\\
        0 & 0 & 1
    \end{pmatrix}, \ R_{0}^{(2)}=\begin{pmatrix}
        0 & r_2 & 0\\
        \dfrac{1}{r_2} & 0 & 0\\
        0 & 0 & 1
    \end{pmatrix}, \ \text{ and } R_{1}^{(2)}=\begin{pmatrix}
      0 & p_2 & 0\\
        \dfrac{1}{p_2} & 0 & 0\\
        0 & 0 & 1
    \end{pmatrix},
    $$
    where $s_4,r_2,p_2 \in \mathbb{C}^*$.\vspace{0.1cm}
    \item The matrices of the representation $\epsilon_3$ are given by
    $$
S^{(3)}=R_{0}^{(3)}=R_{1}^{(3)}=\begin{pmatrix}
        1 & 0 & 0\\
        \dfrac{1}{r_6} & -1 & r_6\\
        0 & 0 & 1
    \end{pmatrix},$$
    where $r_6 \in \mathbb{C}^*$.\vspace{0.1cm}

    \item The matrices of the representation $\epsilon_4$ are given by
    $$
S^{(4)}=R_{0}^{(4)}=R_{1}^{(4)}=\begin{pmatrix}
        1 & r_2 & 0\\
        0 & -1 & 0\\
        0 & \dfrac{1}{r_2} & 1
    \end{pmatrix},$$
    where $r_2\in \mathbb{C}^*$.
\end{enumerate}
\end{theorem}

\begin{proof}
Since $\epsilon$ is a homogeneous $3$-local representation of $M_2WT_n$, there exist matrices $S$, $R_0$, and $R_1$ in $\mathrm{GL}_3(\mathbb{C})$ such that
$$
\epsilon(s_i)=
\begin{pmatrix}
I_{i-1} & 0 & 0 \\
0 & S & 0 \\
0 & 0 & I_{n-i-1}
\end{pmatrix},$$
$$
\epsilon(\rho_i^{(0)})=
\begin{pmatrix}
I_{i-1} & 0 & 0 \\
0 & R_1 & 0 \\
0 & 0 & I_{n-i-1}
\end{pmatrix},
$$
and
\[ \epsilon(\sigma_{i,2})=
\begin{pmatrix}
I_{i-1} & 0 & 0 \\
0 & R_2 & 0 \\
0 & 0 & I_{n-i-1}
\end{pmatrix}
\]
for every $1\leq i\leq n-1$. 
The matrices $S$, $R_1$, and $R_2$ are given by
\[
S=
\begin{pmatrix}
s_1&s_2&s_3\\
s_4&s_5&s_6\\
s_7&s_8&s_9
\end{pmatrix},
\quad 
R_1=
\begin{pmatrix}
r_1&r_2&r_3\\
r_4&r_5&r_6\\
r_7&r_8&r_9\\
\end{pmatrix},
\quad
\text{and} 
\quad 
R_2=
\begin{pmatrix}
p_1&p_2&p_3\\
p_4&p_5&p_6\\
p_7&p_8&p_9\\
\end{pmatrix},
\]
where all coefficients belong to $\mathbb{C}$ and each matrix is invertible. Since $\epsilon$ defines a representation of $M_2WT_n$, the defining relations of the group must be preserved under $\epsilon$. Moreover, by the locality assumption, it is sufficient to examine only a reduced collection of relations, as the remaining ones produce equivalent systems of equations. Hence, it suffices to impose the relations. Substituting the above matrices into these relations and comparing the corresponding entries produces a large polynomial system involving the variables $s_j,r_j,p_j$. The computations are lengthy and highly technical, so the resulting system was solved using Wolfram Mathematica. This procedure yields precisely the required solutions. Finally, a direct computation confirms that the obtained matrices satisfy all defining relations of $M_2WT_n$, completing the proof.
\end{proof}

We now state a result regarding the reducibility of these classified representations.

\begin{proposition}

    Every non-trivial homogeneous $3$-local representation $\epsilon: M_2WT_n \longrightarrow \mathrm{GL}_{n+1}(\mathbb{C})$ is reducible.
\end{proposition}

\begin{proof}
Every non-trivial homogeneous $3$-local representation  $
\epsilon: M_2WT_n \longrightarrow \mathrm{GL}_{n+1}(\mathbb{C})$
must be one of the representations $\epsilon_j$, for $1\leq j\leq 4$, introduced in Theorem \ref{thm3localM2}. Observe that the column vector $
(1,0,\ldots,0)^T$
is invariant under the right action of the matrices
\[
S^{(1)},\ R_0^{(1)},\ R_1^{(1)},\ S^{(4)},\ R_0^{(4)},\ R_1^{(4)}.
\]
Moreover, the row vector $
(0,\ldots,0,1)$
is invariant under the left action of the matrices
\[
S^{(2)},\ R_0^{(2)},\ R_1^{(2)},\ S^{(3)},\ R_0^{(3)},\ R_1^{(3)}.
\]
Hence, the representations $\epsilon_j$, for $j=1,4$, admit a non-trivial invariant subspace generated by the column vector $(1,0,\ldots,0)^T$, whereas the representations $\epsilon_j$, for $j=2,3$, admit a non-trivial invariant subspace generated by the row vector $(0,\ldots,0,1)$. Therefore, each representation $\epsilon_j$ is reducible, and consequently, $\epsilon$ is reducible as well.
\end{proof}

We now state a result regarding the faithfulness of these classified representations.

\begin{proposition}
The representations $
\epsilon_j: M_2WT_n \longrightarrow \mathrm{GL}_n(\mathbb{C}), \ 1\le j \le 4,$
introduced in Theorem \ref{thm3localM2} are unfaithful.
\end{proposition}

\begin{proof}
The representations $\epsilon_3$ and $\epsilon_4$ are evidently unfaithful, since distinct generators are sent to the same matrix.

For the representations $\epsilon_1$ and $\epsilon_2$, one verifies that the non-trivial elements
\[
(\rho_i^{(0)}s_{i+1})^3,
\quad 1\leq i \leq n-2,
\]
are mapped to the identity matrix. Hence, these representations are also unfaithful.
\end{proof}

\end{document}